\documentclass{amsart}
\usepackage{setspace, amsmath, amsthm, amssymb, amsfonts, amscd, epic, graphicx, ulem, dsfont}
\usepackage[T1]{fontenc}
\usepackage{multirow}
\usepackage{bbm}

\newtheorem{lem}{Lemma}
\newtheorem{thm}{Theorem}
\newtheorem*{un-thm}{Theorem}

\newtheorem{pro}{Proposition}
\newtheorem{cor}{Corollary}

\theoremstyle{abstract}

\theoremstyle{remark}
\newtheorem*{rema}{Remark}

\newtheorem*{exa}{Example}

\theoremstyle{definition}
\newtheorem*{defn}{Definition}

\newcommand{\R}{\mathbb{R}}

\newcommand{\C}{\mathbb{C}}

\makeatletter \@namedef{subjclassname@2010}{
  \textup{2010} Mathematics Subject Classification}
\makeatother

\begin{document}

\title[Commutativity of Unbounded Operators]{Commutativity of Unbounded Normal and Self-adjoint Operators and Applications}
\author{Mohammed Hichem Mortad}

\address{Department of
Mathematics, University of Oran, B.P. 1524, El Menouar, Oran 31000.
Algeria.\newline {\bf Mailing address}:
\newline Dr Mohammed Hichem Mortad \newline BP 7085
Es-Seddikia\newline Oran
\newline 31013 \newline Algeria}

\email{mhmortad@gmail.com, mortad@univ-oran.dz.}

\begin{abstract}
Devinatz, Nussbaum and von Neumann established some important
results on the strong commutativity of self-adjoint and normal
unbounded operators. In this paper, we prove results in the same
spirit.
\end{abstract}



\thanks{Partially supported by "Laboratoire d'Analyse Mathématique et Applications".}

\maketitle

\section{Introduction}

First, we assume that all operators operators are linear. Bounded
operators are assumed to be defined on the whole Hilbert space.
Unbounded operators are supposed to have dense domains, and so they
will be said to be densely defined. For general references on
unbounded operator theory, see \cite{RS1,RUD,SCHMUDG-book-2012,WEI}.

Let us, however, recall some notations that will be met below. If
$A$ and $B$ are two unbounded operators with domains $D(A)$ and
$D(B)$ respectively, then $B$ is called an extension of $A$, and we
write $A\subset B$, if $D(A)\subset D(B)$ and if $A$ and $B$
coincide on $D(A)$.   If $A\subset B$, then $B^*\subset A^*$.

The product $AB$ of two unbounded operators $A$ and $B$ is defined
by
\[BA(x)=B(Ax) \text{ for } x\in D(BA)\]
where
\[D(BA)=\{x\in D(A):~Ax\in D(B)\}.\]

Recall too that the unbounded operator $A$, defined on a Hilbert
space $H$, is said to be invertible if there exists an
\textit{everywhere defined} (i.e. on the whole of $H$) bounded
operator $B$, which then will be designated by $A^{-1}$, such that
\[BA\subset AB=I\]
where $I$ is the usual identity operator. This is the definition
adopted in the present paper. It may be found in e.g. \cite{Con} or
\cite{GGK}.

An unbounded operator $A$ is said to be closed if its graph is
closed; symmetric if $A\subset A^*$; self-adjoint if $A=A^*$ (hence
from known facts self-adjoint operators are automatically closed);
normal if it is \textit{closed} and $AA^*=A^*A$ (this implies that
$D(AA^*)=D(A^*A)$). It is also worth recalling that normal operators
$A$ do obey the domain condition $D(A)=D(A^*)$.

Commutativity of unbounded operators must be handled with care.
First, recall the definition of two strongly commuting unbounded
(self-adjoint) operators (see e.g. \cite{RS1}):

\begin{defn}
Let $A$ and $B$ be two unbounded self-adjoint operators. We say that
$A$ and $B$ strongly commute if all the projections in their
associated projection-valued measures commute.

This in fact equivalent to saying that
$e^{itA}e^{isB}=e^{isB}e^{itA}$ for all $s,t\in\R$.
\end{defn}

We shall use the same definition for unbounded normal operators.

Nelson \cite{Nelson-Anal-vectors} showed that there exists a pair of
two essentially self-adjoint operators $A$ and $B$ on some common
domain $D$ such that
\begin{enumerate}
  \item $A:D\rightarrow D$, $B:D\rightarrow D$,
  \item $ABx=BAx$ for all $x\in D$,
  \item but $e^{it\overline{A}}$ and $e^{is\overline{B}}$ do not commute, i.e. $A$ and $B$
  do  not strongly commute.
\end{enumerate}

Based on the previous example, Fuglede \cite{FUG-19825} proved a
similar result. Hence an expression of the type $AB=BA$, although
being quite strong since it implies that $D(AB)=D(BA)$, does not
necessarily mean that $A$ and $B$ strongly commute.

There are results (see e.g. \cite{FUG-19825},
\cite{Nelson-Anal-vectors}) giving conditions implying the strong
commutativity of $\overline{A}$ and $\overline{B}$. For instance, we
have
\begin{thm}[\cite{Nussbaum-PAMS-semibounded-1997}]
Let $A$ and $B$ be two semi-bounded operators in a Hilbert space $H$
and let $D$ be a dense linear manifold contained in the domains of
$AB$, $BA$, $A^2$ and $B^2$ such that $ABx=BAx$ for all $x\in D$. If
the restriction of $(A+B)^2$ to $D$ is essentially self-adjoint,
then $A$ and $B$ are essentially self-adjoint and $\overline{A}$ and
$\overline{B}$ strongly commute.
\end{thm}

Now, let us recall results obtained by Devinatz-Nussbaum (and von
Neumann) on strong commutativity:

\begin{thm}[Devinatz-Nussbaum-von Neumann, \cite{DevNussbaum-von-Neumann} and cf. \cite{Nussbaum-TAMS-commu-unbounded-normal-1969}]\label{Devinatz-Nussbaum-von Neumann} If there exists a
self-adjoint operator $A$ such that $A\subseteq BC$, where $B$ and
$C$ are self-adjoint, then $B$ and $C$ strongly commute.
\end{thm}

\begin{cor}\label{Devinatz-Nussbaum-von Neumann: T=T1T2}
Let $A$, $B$ and $C$ be unbounded self-adjoint operators. Then
\[A\subseteq BC \Longrightarrow A=BC.\]
\end{cor}

The following improvement of Corollary \ref{Devinatz-Nussbaum-von
Neumann: T=T1T2} appeared in
\cite{Nussbaum-TAMS-commu-unbounded-normal-1969}

\begin{cor}\label{nussbaum-improvement-devintaz}
Let $A$ be an unbounded self-adjoint operator and let $B$ and $C$ be
two closed symmetric operators such that $AB\subset C$. If $B$ has a
bounded inverse (hence it is self-adjoint), then $C$ is
self-adjoint. Besides $AB=C$.
\end{cor}

The research work on the normality of the product of two bounded
normal operators started in 1930 by the work of Gantmaher-Krein (see
\cite{Gantmaher-Krein}). Then, it followed papers by Weigmann
(\cite{Wieg-matrices-normal,Wieg-infinite-matrices}) and Kaplansky
\cite{Kapl}. Gheondea \cite{Gheondea} quoted that "the normality of
operators in the Pauli algebra representations became of interest in
connection with some questions in polarization optics" (see
\cite{Tudor-Gheondea}). Similar problems also arise in Quantum
Optics (see \cite{Barnett-Radmore}).

There have been several successful attempts by the author to
generalize the previous to the case where at least one operator is
unbounded. See e.g. \cite{Mortad-Demm-math} and
\cite{Mortad-Concrete-operators-2012}. One of the important
considerations of the normality of the product of unbounded normal
operators (NPUNO, in short) is strong commutativity of the latter.
Indeed, the following striking result (which is not known to many)
shows the great interest of investigating the question of (NPUNO):

\begin{thm}[Devinatz-Nussbaum, \cite{DevNussbaum}]\label{Devinatz-Nussbaum}If $A$, $B$ and $N$ are unbounded normal operators
obeying $N=AB=BA$, then $A$ and $B$ strongly commute.
\end{thm}

In this paper we weaken the condition $AB=BA$ to $AB\subset BA$,
say, and still derive results on the strong commutativity of $A$ and
$B$ (in unbounded normal and self-adjoint settings).

When proving the normality of a product, we need its closedness and
its adjoint. Let us thus recall known results on those two notions:

\begin{thm}\label{adjoints and closedness A+B AB basic}
Let $A$ be a densely defined unbounded operator.
\begin{enumerate}
  \item $(BA)^*=A^*B^*$ if $B$ is bounded.
  \item $A^*B^*\subset (BA)^*$ for any densely unbounded $B$ and if $BA$ is
  densely defined.
  \item Both $AA^*$ and $A^*A$ are self-adjoint whenever $A$ is
  closed.
\end{enumerate}
\end{thm}

\begin{lem}[\cite{WEI}]\label{(AB)*=B*A*} If $A$ and $B$ are densely defined and $A$ is invertible with
inverse $A^{-1}$ in $B(H)$, then $(BA)^* =A^* B^*$.
\end{lem}

\begin{lem}\label{AB closed}The product $AB$ (in this order) of
two densely defined closed operators $A$ and $B$ is closed if one of
the following occurs:
\begin{enumerate}
  \item $A$ is invertible,
  \item $B$ is bounded.
\end{enumerate}
\end{lem}

For related work on strong commutativity, we refer the reader to
\cite{Arai-2005-Rev-math-phys,Nussbaum-TAMS-commu-unbounded-normal-1969,Nussbaum-PAMS-semibounded-1997,Schmudgen-strong-comm-SA-Unbd-PAMS,Zagorodnyuk}.
For similar papers on products, the interested reader may consult
 \cite{Gesztesy}, \cite{Jorgensen-PALLE}, \cite{Mortad-PAMS2003},
\cite{Mortad-IEOT-2009}, \cite{Mortad-Demm-math},
\cite{Schmudgen-Friedrich-II} and \cite{Seb-Stochel},
 and further bibliography
cited therein.

\section{Main Results}

We start by giving a result on strong commutativity of unbounded
normal operators. But we first have the following result on (NPUNO):

\begin{thm}\label{normality of AB}
Let $A$ and $B$ be two unbounded normal operators verifying
$AB\subset BA$. If $B$ is invertible, then $BA$ and $\overline{AB}$
are both normal whenever $AB$ is densely defined.
\end{thm}

\begin{proof}
Since $B$ is invertible,
\[AB\subset BA \Longrightarrow A\subset BAB^{-1}\Longrightarrow B^{-1}A\subset AB^{-1}.\]
By the Fuglede theorem (\cite{FUG}), we obtain
\[B^{-1}A^*\subset A^*B^{-1}.\]
Left multiplying, then right multiplying by $B$ yield
\[A^*B\subset BA^* \text{ so that } AB^*\subset B^*A\]
by Lemma \ref{(AB)*=B*A*}. Next
\[(BA)^*BA\subset (AB)^*BA=B^*A^*BA\subset B^*BA^*A.\]
But $BA$ is closed, hence $(BA)^*BA$ is self-adjoint. Since $B^*B$
and $A^*A$ are also self-adjoint, Corollary
\ref{Devinatz-Nussbaum-von Neumann: T=T1T2} gives us
\[(BA)^*BA=B^*BA^*A.\]
Very similar arguments may be applied to prove that
\[BA(BA)^*=BB^*AA^*.\]
Thus, and since $A$ and $B$ are normal, we obtain
\[(BA)^*BA=BA(BA)^*,\]
establishing the normality of $BA$.

To prove that $\overline{AB}$ is normal, observe first that thanks
to the invertibility of $B$, we have
\[AB\subset BA\Longrightarrow A^*B^*\subset B^*A^*.\]
Since $B$ is invertible, $B^*$ too is invertible. Thus by the first
part of the proof, $B^*A^*$ is normal and so is $(AB)^*=B^*A^*$.
Hence its adjoint $(AB)^{**}=\overline{AB}$ stays normal.
\end{proof}

The hypothesis $AB\subset BA$ is fundamental as seen in the
following example:

\begin{exa}
Let $A$ and $B$ be defined respectively by
\[Af(x)=f'(x) \text{ and } Bf(x)=e^{|x|}f(x)\]
on
\[D(A)=H^1(\R) \text{ and } D(B)=\{f\in L^2(\R):~e^{|x|}f\in L^2(\R)\},\]
where $H^1(\R)$ is the usual Sobolev space.

The operator $A$ is known to be normal because it is unitarily
equivalent (via the $L^2(\R)$-Fourier transform) to a multiplication
operator by a complex-valued function.

As for $B$, it is clearly densely defined, self-adjoint and
invertible. It is also easy to see that $AB$ and $BA$ do not
coincide on any dense set.

Finally, let us show that $N:=BA$ (which is obviously closed) is
\textit{not normal}. We have
\[Nf(x)=e^{|x|}f'(x)\text{ defined on } D(N)=D(BA) \text{ hence } \]
\[D(N)=\{f\in L^2(\R):f'\in L^2(\R), e^{|x|}f'\in L^2(\R)\}=\{f\in L^2(\R): e^{|x|}f'\in L^2(\R)\}.\]
To compute $N^*$, the adjoint of $N$, we should do it first for
$C_0^{\infty}(\R^*)$ functions, then proceed as in
\cite{Mortad-PAMS2003}. We find that
\[N^*f(x)=e^{|x|}(\mp f(x)-f'(x))\]
on
\[D(N^*)=\{f\in L^2(\R): e^{|x|}f\in L^2(\R), e^{|x|}f'\in L^2(\R)\}.\]

Then we easily obtain that

\[NN^*f(x)=e^{2|x|}(-f(x)\mp2f'(x)-f''(x))\]

and

\[N^*Nf(x)=e^{2|x|}(\mp2f'(x)-f''(x)),\]

establishing the non-normality of $N$.
\end{exa}

\begin{cor}\label{normality AB subset BA B invertible equality closure AB=BA
COROLLARY!!!!!!!} Let $A$ and $B$ be two unbounded normal operators
verifying $AB\subset BA$. If $B$ is invertible, then
$\overline{AB}=BA$ whenever $AB$ is densely defined.
\end{cor}

\begin{proof}
Since $BA$ is closed, we have
\[AB\subset BA \Longrightarrow \overline{AB}\subset BA.\]
But both $\overline{AB}$ and $BA$ are normal, and normal operators
are maximally normal (see e.g. \cite{RUD}), hence
\[\overline{AB}=BA.\]
\end{proof}

\begin{cor}
Let $A$ and $B$ be two unbounded normal operators. If $B$ is
invertible and $BA=AB$, then $A$ and $B$ strongly commute whenever
$AB$ is densely defined.
\end{cor}

\begin{proof}By Theorem \ref{normality of AB}, $BA$ and this time
$AB$ is normal too. By Theorem \ref{Devinatz-Nussbaum}, $A$ and $B$
strongly commute.
\end{proof}

\begin{rema}
We could have stated the previous corollary as: \textit{Let $A$ and
$B$ be two unbounded normal operators. If $B$ is invertible and
$AB\subset BA$ and $AB$ is closed, then $A$ and $B$ strongly commute
whenever $AB$ is densely defined.}
\end{rema}

As an application of the strong commutativity of unbounded normal
operators, we have the following result (cf.
\cite{mortad-CAOT-sum-normal}):

\begin{pro}\label{sum normal strongly commuting}
Let $A$ and $B$ be two strongly commuting unbounded normal
operators. Then $A+B$ is essentially normal.
\end{pro}

\begin{rema}
Recall that an unbounded closeable operator is said to be
essentially normal if it has a normal closure (of course, this
terminology has a different signification in Banach algebras).
\end{rema}

\begin{proof}
Since $A$ and $B$ are normal, by the spectral theorem we may write
\[A=\int_{\C}zdE_A(z) \text{ and } B=\int_{\C}z'dF_B(z'),\]
where $E_A$ and $F_B$ designate the associated spectral measures. By
the strong commutativity, we have
\[E_A(I)F_B(J)=F_B(J)E_A(I)\]
for all Borel sets $I$ and $J$ in $\C$. Hence
\[E_{A,B}(z,z')=E_A(z)F_B(z')\]
defines a two parameter spectral measure. Thus
\[C=\int_{\C}\int_{\C}(z+z')dE_{A,B}(z,z')\]
defines a normal operator, such that $C=\overline{A+B}$. Therefore,
$A+B$ is essentially normal.
\end{proof}

\begin{cor}
Let $A$ and $B$ be two unbounded normal operators. If $B$ is
invertible and $BA=AB$ (where $AB$ is densely defined). Then $A+B$
is essentially normal.
\end{cor}

Next, we treat the case of two unbounded self-adjoint operators. We
start with the following interesting result

\begin{pro}\label{strong commutativity AB subset BA B invertible}
Let $A$ and $B$ be two unbounded self-adjoint operators such that
$B$ is invertible. If $AB\subset BA$, then $A$ and $B$ strongly
commute whenever $AB$ is densely defined.
\end{pro}

\begin{proof}
By Theorem \ref{normality of AB} and Corollary \ref{normality AB
subset BA B invertible equality closure AB=BA COROLLARY!!!!!!!}, we
have $\overline{AB}=BA$. Hence
\[(BA)^*=[\overline{AB}]^*=(AB)^*=B^*A^*=BA.\]
That is $BA$ is self-adjoint. Whence and since $AB\subset BA$,
Corollary \ref{nussbaum-improvement-devintaz} yields $AB=BA$. Thus,
and by Theorem \ref{Devinatz-Nussbaum-von Neumann}, $A$ and $B$
strongly commute.
\end{proof}

We have a similar result for self-adjoint operators to that of
Proposition \ref{sum normal strongly commuting}

\begin{cor} [cf. \cite{Arai-2005-Rev-math-phys}]
Let $A$ and $B$ be two unbounded self-adjoint operators such that
$B$ is invertible. If $AB\subset BA$, then $A+B$ is essentially
self-adjoint. Moreover, for all $t\in\R$:
\[e^{it\overline{(A+B)}}=e^{itA}e^{itB}=e^{itB}e^{itA}.\]
\end{cor}

\begin{proof}By Proposition \ref{strong commutativity AB subset BA B
invertible}, $A$ and $B$ strongly commute. Then, apply an akin proof
to that of Proposition \ref{sum normal strongly commuting}. For the
last displayed equation just use the Trotter product formula which
may be found in \cite{RS1}.
\end{proof}

In \cite{DevNussbaum} it was noted that Theorem
\ref{Devinatz-Nussbaum-von Neumann} does not have an analog in the
case of unbounded normal operators. The counterexample is very
simple, it suffices to take the product of two non-commuting unitary
operators which is unitary anyway. Nonetheless, we have the
following maximality result

\begin{thm}\label{maximality AB, C C normal all unboudned}
Let $A$ and $B$ be two unbounded self-adjoint operators such that
$B$ is positive and invertible. Let $C$ be an unbounded normal
operator. If $AB\subset C$, then $A$ and $B$ strongly commute, $AB$
is self-adjoint (whenever it is densely defined) and hence $AB=C$.
\end{thm}

To prove it we need the following result:

\begin{lem}\label{lemma commuting bd function unbd}
Let $S$ and $T$ be two self-adjoint operators defined on a Hilbert
space $H$. Assume that $S$ is bounded which commutes with $T$, i.e.
$ST\subset TS$. Then $f(S)T\subset Tf(S)$ for any real-valued
continuous function $f$ defined on $\sigma(S)$. In particular, we
have $S^{\frac{1}{2}}T\subset TS^{\frac{1}{2}}$, if $S$ is positive.
\end{lem}

\begin{proof}
We may easily show that if $P$ is a real polynomial, then
$P(S)T\subset TP(S)$.

Now let $f$ be a continuous function on $\sigma(S)$. Then (by the
density of the polynomials defined on the compact $\sigma (S)$ in
the set of continuous functions with respect to the supremum norm)
there exists a sequence $(P_n)$ of polynomials converging uniformly
to $f$ so that
\[\lim_{n\rightarrow \infty}\|P_n(S)-f(S)\|_{B(H)}=0.\]
Let $x\in D(f(S)T)=D(T)$. Let $t\in D(T)$ and $x=f(S)t$. Setting
$x_n=P_n(S)t$, we see that
\[Tx_n=TP_n(S)t=P_n(S)Tt\longrightarrow f(S)Tt.\]
Since $x_n\rightarrow x$, by the closedness of $T$, we get
\[x=f(S)t\in D(T) \text{ and } Tx=Tf(S)t=f(S)Tt,\]
that is, we have proved that $f(S)T\subset Tf(S)$.
\end{proof}

Now we give the proof of Theorem \ref{maximality AB, C C normal all
unboudned}:

\begin{proof}
Since $AB\subset C$ and $B$ is invertible, we get
\[C^*\subset (AB)^*=B^*A^*=BA \text{ or } B^{-1}C^*\subset A.\]
It is also clear that $AB\subset C$ implies that $A\subset CB^{-1}$.
Hence, and since $C$ is normal, the Fuglede-Putnam theorem (see
\cite{FUG} and \cite{PUT}) allow us to write
\[B^{-1}C^*\subset CB^{-1}\Longrightarrow B^{-1}C\subset C^*B^{-1}.\]
So
\[B^{-1}AB\subset BAB^{-1}.\]
Left multiplying but $B^{-1}$, then right multiplying by $B^{-1}$
give us
\[(B^{-1})^2A\subset A(B^{-1})^2.\]

Since $B^{-1}$ is bounded and positive, Lemma \ref{lemma commuting
bd function unbd} allows us to say that $B^{-1}$ and $A$ commute.
Hence
\[B^{-1}A\subset AB^{-1} \text{ or just } AB\subset BA.\]

By Proposition \ref{strong commutativity AB subset BA B invertible},
$BA$ is self-adjoint, and $A$ strongly commutes with $B$. Therefore,
by Corollary \ref{nussbaum-improvement-devintaz}, we have $AB=BA$.
Thus, and since self-adjoint operators are maximally normal, we
deduce that $AB=C$.
\end{proof}

\end{document}